\documentclass[reqno,oneside]{amsart}
\usepackage{amsmath, amsfonts, amssymb,amsbsy,eucal,mathrsfs}

\usepackage[all]{xy}

\usepackage{chngcntr}

\usepackage{todonotes}

\numberwithin{equation}{section}
\newtheorem{thm}{Theorem}[section]

\newtheorem{lemma}[thm]{Lemma}
\newtheorem{prop}[thm]{Proposition}
\newtheorem{cor}[thm]{Corollary}

\newtheorem{rmk}[thm]{Remark}

\DeclareMathOperator{\G}{G}    
\DeclareMathOperator{\IG}{IG}
\DeclareMathOperator{\Sp}{Sp}
\DeclareMathOperator{\QH}{QH}
\DeclareMathOperator{\BQH}{BQH}

\newcommand{\D}{{\Delta}}
\newcommand{\ZZ}{{\mathbb Z}}
\newcommand{\QQ}{{\mathbb Q}}
\newcommand{\bQ}{{\mathbb Q}}
\newcommand{\starz}{{\!\ \star_{0}\!\ }}
\newcommand{\pt}{{{\rm pt}}}
\newcommand{\C}{{\mathbb C}}
\newcommand{\p}{\mathbb{P}}

\def \Yo {{\mathring{Y}}}
\def \spec {{{\rm Spec}}}

\newcommand{\cA}{\mathcal{A}}
\newcommand{\cB}{\mathcal{B}}

\begin{document}

\title{Update on quantum cohomology of $\IG(2,2n)$}

\date{}

\author{Anton Mellit}
\address{Scuola Internazionale Superiore di Studi Avanzati (SISSA), 
via Bonomea 265, Trieste 34136, Italy}
\email{mellit@gmail.com}

\author{Nicolas Perrin}
\address{Universit\'e de Versailles Saint-Quentin-en-Yvelines, 
Laboratoire de Math\'ematiques de Versailles, 
45 avenue des Etats-Unis, 
78035 Versailles Cedex, France}
\email{nicolas.perrin@uvsq.fr}

\author{Maxim Smirnov}
\address{
Riemann Center for Geometry and Physics, 
Leibniz Universit\"at Hannover,
Welfengarten 1, 
30167 Hannover, 
Germany
}
\email{msmirnov@ictp.it}


\begin{abstract}
We give another proof of the generic semisimplicity of the big quantum cohomology of the symplectic isotropic Grassmannians $\IG(2,2n)$.
\end{abstract}

\maketitle

\tableofcontents

\section*{Introduction}

In papers \cite{GMS, Pe} jointly with Sergey Galkin we have shown that the big quantum cohomology of the symplectic isotropic Grassmannians $\IG(2,2n)$ is generically semisimple. In \cite{GMS} only the case of $\IG(2,6)$ is considered and the proof is computer-assisted. The proof in \cite{Pe} is not computer-assisted and works for all $\IG(2,2n)$, but it needs many lengthy computations.  The purpose of this paper is to give a pure thought argument which works uniformly for all $\IG(2,2n)$, is conceptually different from the ones in \textit{loc.cit.}, and might be generalizable to all $\IG(m,2n)$.

The strategy of our proof is the following. We view the big quantum cohomology as a deformation family of zero-dimensional schemes over a base (see Sec.~\ref{SubSec.: Def. Picture}). The special fibre of this family is (the spectrum of) the small quantum cohomology. Generic semisimplicity of the big quantum cohomology translates into the smoothness of the generic fibre\footnote{By "generic fiber" we always mean "fiber over the generic point of the base scheme".} of the deformation family. Similarly, generic semisimplicity of the small quantum cohomology translates into the smoothness of the special fibre. So it may very well happen that the generic fibre is smooth, whereas the special one is not. This is exactly what happens in the case of $\IG(2,2n)$. In fact, we prove that for $\IG(2,2n)$ the total space of the deformation family (the spectrum of the big quantum cohomology) is a regular scheme over a field of characteristic zero. Therefore, by a version of "generic smoothness" in characteristic zero, one concludes that the generic fibre is smooth. Thus, we obtain the generic semisimplicity of the big quantum cohomology,\footnote{This argument reminded us of Problem 2.8 from \cite{HeMaTe}.} whereas the small quantum cohomology of $\IG(2,2n)$ is known to be non-semisimple by \cite{ChMaPe,ChPe}.

Our original motivation to look at this problem comes from Dubrovin's conjecture. We will not elaborate on this here and rather refer to \cite{GaGoIr,GMS,HeMaTe} and references therein. Let us just mention that the Grassmannians $\IG(2,2n)$ are the simplest explicit examples available in the literature where one has to work with big quantum cohomology to formulate Dubrovin's conjecture.

\medskip

\noindent \textbf{Acknowledgements.} We would like to thank Tarig Abdelgadir, Erik Carlsson, Roman Fedorov, Alexander Kuznetsov, Christian Lehn, Yuri Manin, Sina T\"ureli and Runako Williams for valuable discussions, comments, and attention to this work. The first and the third named authors are very grateful to the International Centre for Theoretical Physics (ICTP) in Trieste for financial support and excellent working conditions. The third named author thanks the Riemann Center for Geometry and Physics at the Leibniz Universit\"at Hannover for support at the final stage of this project.

\section{Conventions and notation for quantum cohomology}

Here we briefly recall the definition of the quantum cohomology ring for a smooth projective variety $X$. To simplify the exposition and avoid introducing unnecessary notation we impose from the beginning the following conditions on $X$: it is a Fano variety of Picard rank 1 and $H^{odd}(X,\QQ)=0$. For a thorough introduction we refer to \cite{Ma}.

\subsection{Definition}

Let us fix a graded basis $\Delta_0, \dots , \Delta_s$ in $H^*(X, \QQ)$ and dual linear coordinates $t_0, \dots, t_s$. It is customary to choose $\Delta_0=1$. Let $R$ be the ring of formal power series $\QQ[[q]]$, $k$ its field of fractions, and $K$ an algebraic closure of $k$.

The genus zero Gromov-Witten potential of $X$ is an element $\Phi \in R[[t_0, \dots, t_s]]$ defined by the formula
\begin{align}\label{Eq.: GW potential}
&\Phi = \sum_{(i_0, \dots , i_s)}  \langle \Delta_0^{\otimes i_0}, \dots, \Delta_s^{\otimes i_s} \rangle \frac{t_0^{i_0} \dots t_s^{i_s} }{i_0!\dots i_s!}, 
\end{align}
where 
$\langle \Delta_0^{\otimes i_0}, \dots, \Delta_s^{\otimes i_s} \rangle := \sum_{d=0}^{\infty} \langle \Delta_0^{\otimes i_0}, \dots, \Delta_s^{\otimes i_s} \rangle_d q^d$,
and the rational numbers $\langle \Delta_0^{\otimes i_0}, \dots, \Delta_s^{\otimes i_s} \rangle_d$ are Gromov-Witten invariants of $X$ of degree $d$.

Using \eqref{Eq.: GW potential} one defines the quantum cohomology ring of $X$. Namely, on the basis elements we put
\begin{align}\label{Eq.: Quantum cohomology}
\Delta_a \star \Delta_b = \sum_c \Phi_{abc} \Delta^c,
\end{align}
where $\Phi_{abc}=\frac{\partial^3 \Phi}{\partial t_a \partial t_b \partial t_c}$, and $\D^0, \dots, \D^s$ is the basis dual to  $\D_0, \dots, \D_s$ with respect to the Poincar\'e pairing. Expression \eqref{Eq.: Quantum cohomology} is naturally interpreted as an element of $H^*(X, \bQ)\otimes_{\bQ} K[[t_0, \dots, t_s]]$.

\smallskip

It is well known that \eqref{Eq.: Quantum cohomology} makes $H^*(X, \bQ)\otimes_{\bQ} K[[t_0, \dots, t_s]]$ into a commutative, associative, graded $K[[t_0, \dots, t_s]]$-algebra with the identity element $\D_0$. We will denote this algebra $\BQH(X)$. For convenience we recall the definition of the grading:
\begin{align*}
& \deg(\D_i)=|\D_i|, \quad \deg(q)= \text{index} \, (X), 
\quad \deg(t_i)=1-|\D_i|,
\end{align*} 
where $|\D_i|$ is the Chow degree of $\D_i$ and $\text{index} \, (X)$ is the largest integer  $n$ such that $-K_X = nH$ for some ample divisor $H$ on $X$. 

The algebra $\BQH(X)$ is called the \textit{big quantum cohomology algebra} of $X$ to distinguish it from a simpler object called the \textit{small quantum cohomology algebra} which is the quotient of $\BQH(X)$ with respect to the ideal $(t_0, \dots, t_s)$. We will denote the latter $\QH(X)$ and use $\starz$ instead of $\star$ for the product in this algebra. It is a finite dimensional $K$-algebra. Equivalently one can say that $\QH(X)=H^*(X,\bQ) \otimes_{\bQ} K = H^*(X,K)$ as a vector space, and the $K$-algebra structure is defined by putting $\Delta_a \starz \Delta_b = \sum_c \langle \D_a, \D_b, \D_c \rangle \Delta^c$.

\subsubsection{Remark}

We are using a somewhat non-standard notation $\BQH(X)$ for the big quantum cohomology and $\QH(X)$ for the small quantum cohomology to stress the difference between the two. Note that this notation is different from the one used in \cite{GMS} and is closer to the notation of \cite{Pe}.

\subsubsection{Remark} 
\label{SubSubSec.: Remark on difference of notation with Manin's book}

The above definitions look slightly different from the ones given in~\cite{Ma}. The differences are of two types. The first one is that $\QH(X)$ and $\BQH(X)$ are in fact defined already over the ring $R$ and not only over $K$. We pass to $K$ from the beginning, since in this paper we are only interested in generic semisimplicity of quantum cohomology. The second difference is that in some papers on quantum cohomology one unifies the coordinate $q$ with those coordinates $t_i$ which are dual to $H^2(X, \bQ)$, but the resulting structures are equivalent.

\subsection{Deformation picture}
\label{SubSec.: Def. Picture}

The small quantum cohomology, if considered over the ring $R$ (cf. Remark~\ref{SubSubSec.: Remark on difference of notation with Manin's book}), is a deformation of the ordinary cohomology algebra, i.e. if you put $q=0$, then the quantum product becomes the ordinary cup-product. Similarly, the big quantum cohomology is an even bigger deformation family of algebras. Since we work not over $R$ but over $K$, we lose the point of classical limit but still retain the fact that $\BQH(X)$ is a deformation family of algebras with the special fiber being $\QH(X)$. 

Throughout this paper we view $\spec(\BQH(X))$ as a deformation family of zero-dimensional schemes over $\spec (K[[t_0,\dots, t_s]])$. In the base of the deformation we consider the following two points: the origin (the closed point given by the maximal ideal $(t_0, \dots , t_n)$) and the generic point $\eta$. The fiber of this family over the origin is the spectrum of the small quantum cohomology $\spec(\QH(X))$. The fiber over the generic point will be denoted by $\spec(\BQH(X)_{\eta})$. It is convenient to summarize this setup in the diagram
\begin{align}\label{Eq.: BQH family}
\xymatrix{
\spec(\QH(X)) \ar[d] \ar[r]   &    \spec(\BQH(X)) \ar[d]^{\pi} & \ar[l] \spec(\BQH(X)_{\eta}) \ar[d]^{\pi_{\eta}} \\
\spec (K)    \ar[r]      & \spec (K[[t_0, \dots, t_s]])  &   \ar[l] \eta
}
\end{align}
where both squares are Cartesian. 
    
By construction $\BQH(X)$ is a free module of finite rank over $K[[t_0, \dots,t_s]]$. Therefore, it is a noetherian semi-local $K$-algebra which is flat and finite over $K[[t_0, \dots,t_s]]$. Note that neither $K[[t_0, \dots,t_s]]$ nor $\BQH(X)$ are finitely generated over the ground field $K$. Therefore, some extra care is required in the standard commutative algebra (or algebraic geometry) constructions. For example, the notion of smoothness is one of such concepts.

\subsection{Semisimplicity}
\label{SubSec.: Semisimplicity}

Let $A$ be a finite dimensional algebra over a field $F$ of characteristic zero. It is called \textit{semisimple} if it is a product of fields. Equivalently, the algebra $A$ is semisimple iff the scheme $\spec(A)$ is reduced. Another equivalent condition is to require the morphism $\spec(A) \to \spec(F)$ to be smooth.

\smallskip

\textbf{Definition:} one says that $\BQH(X)$ is \textit{generically semisimple} iff $\BQH(X)_{\eta}$ is a semisimple algebra.

\smallskip

As already described in the introduction, we will prove the generic semisimplicity of $\BQH(\IG(2,2n))$ by a version of generic smoothness in characteristic zero.

\section{Geometry of $\IG(2,2n)$}

Let $V$ be a complex vector space endowed with a symplectic form $\omega$. In this case the dimension of $V$ has to be even and we denote it by $2n$. For any $1 \leq m \leq n$ there exists an algebraic variety $\IG_{\omega}(m, V)$ that parametrizes $m$-dimensional isotropic subspaces of $V$. For $m=2$, and this is the case we are considering in this paper, it has the following explicit description. Consider the ordinary Grassmannian $\G(2,V)$ with its Pl\"ucker embedding into $\mathbb{P}(\Lambda^2V)$. The symplectic form $\omega$ defines a hyperplane $H_{\omega} \subset \mathbb{P}(\Lambda^2V)$ and the intersection of $\G(2,V)$ with $H_{\omega}$ is exactly $\IG_{\omega}(2, V)$. Thus, we have inclusions
\begin{align}\label{Eq.: Hyperplane section diagram for IG(2,2n)}
\IG_{\omega}(2, V) \subset \G(2,V) \subset  \mathbb{P}(\Lambda^2V).
\end{align}
For different $\omega$'s the varieties $\IG_{\omega}(m, V)$ are isomorphic, therefore we will often simply write $\IG(m, 2n)$.

\subsection{Special cohomology classes}

As for ordinary Grassmannians, one considers the short exact sequence of vector bundles on $X=\IG_{\omega}(2, V)$
\begin{align}\label{Eq.: Tautological S.E.S I}
0 \to \mathcal{U} \to \mathcal{V} \to \mathcal{V/U} \to 0,
\end{align}
where $\mathcal{V}$ is the trivial vector bundle with fiber $V$, $\mathcal{U}$ is the subbundle of isotropic subspaces, and $\mathcal{V/U}$ is the quotient bundle. Usually one refers to $\mathcal{U}$ and $\mathcal{V/U}$ as \textit{tautological subbundle} and \textit{tautological quotient bundle} respectively. Apart from \eqref{Eq.: Tautological S.E.S I} one can also consider the short exact sequence
\begin{align}\label{Eq.: Tautological S.E.S II}
0 \to \mathcal{U}^{\perp}/\mathcal{U} \to \mathcal{V}/\mathcal{U} \to \mathcal{V}/\mathcal{U}^\perp \to 0.
\end{align}
By taking Chern classes of vector bundles in these sequences we obtain two sets of cohomology classes which generate the cohomology ring:  

\smallskip

\textit{i)} \textit{Special Schubert classes}. These are the Chern classes of $\mathcal{V/U}$. They are indexed by integers $k \in [0,2n-2]$ and can be explicitly described as follows. Let
$$Z(E_{2n - k - 1}) = \{V_2 \in X \ | \ \dim(V_2 \cap E_{2n - k - 1}) \geq 1 \}.$$
Then $\sigma_k = [Z(E_{2n - k - 1})] \in H^{2k}(X,\ZZ)$. In the above, $E_{2n - k - 1}$ is a subspace of dimension $2n - k -1$ such that the rank of $\omega\vert_{E_{2n - k - 1}}$ is minimal. Note that $\sigma_0 = 1$ is the fundamental class of $X$. 

\smallskip

\textit{ii)} These are the Chern classes of $\mathcal{U}$ and $\mathcal{U}^{\perp}/\mathcal{U}$. Vector bundle $\mathcal{U}$ is of rank $2$, so it only has two non-vanishing Chern classes $a_i=c_i(\mathcal{U})$ for $i=1, 2$. Vector bundle $\mathcal{U}^{\perp}/\mathcal{U}$ is self-dual of rank $2n-4$, therefore it has only $n-2$ non-vanishing Chern classes $b_i= c_{2i}(\mathcal{U}^{\perp}/\mathcal{U})$ for $i\in [1, n- 2]$.

\subsection{Cohomology ring of $\IG(2,2n)$}
\label{SubSec.: Cohomology of IG}

The cohomology ring of $X$ can be neatly described in terms of generators and relations. We will give two presentations using the two sets of special cohomology classes defined above.

\begin{prop}
\label{Prop.: Presentation for coh I}

The cohomology ring $H^*(X,\QQ)$ is isomorphic to the quotient of the ring $\QQ[\sigma_1,\cdots,\sigma_{2n-2}]$ by the ideal generated by the elements
\begin{align}\label{Eq.: Eqns for H(IG), part I}
\det(\sigma_{1 + j - i})_{1 \leq i,j \leq r}, \ \ \ \textrm{with } r \in [3,2n-2]
\end{align}
and the two elements 
\begin{align}\label{Eq.: Eqns for H(IG), part II}
\sigma_{n-1}^2 + 2 \sum_{i = 1}^{n - 1} (-1)^i \sigma_{n - 1 + i} \sigma_{n - 1 - i}  \ \ \textrm{  and  }\ \  \sigma_n^2 + 2 \sum_{i = 1}^{n - 2} (-1)^i \sigma_{n + i} \sigma_{n - i}.
\end{align}
The dimension of $H^*(X, \QQ)$ is equal to $2^2\binom{n}{2}=2n(n-1)$.
\end{prop}

\medskip

\begin{proof}
This is the statement of \cite[Theorem 1.2]{BKT}.
\end{proof}

\begin{prop} 
\label{Prop.: Presentation for coh II}
The cohomology ring $H^*(X,\QQ)$ is isomorphic to the quotient of the ring $\QQ[a_1,a_2,b_1,\cdots,b_{n-2}]$ by the ideal generated by
\begin{align}\label{Eq.: Relation in Presentation II}
(1 - (2a_2 - a_1^2) x^2 + a_2^2 x^4)(1 + b_1 x^2 + \cdots + b_{n-2} x^{2n-4}) = 1
\end{align}
The last equality is viewed as an equality of polynomials in the variable $x$ and gives a concise way to write a system of  equations in the variables $a_i, b_i$.
\end{prop}

\begin{proof} 

This result is well know to specialists but we include a short proof for the convenience of the reader.

Let us start by checking that \eqref{Eq.: Relation in Presentation II} holds in the cohomology ring. Define $P(x)= 1 + a_1 x + a_2 x^2$ and $Q(x)= 1 + b_1 x^2 + \dots + b_{n-2}x^{2n-4}$ and rewrite \eqref{Eq.: Relation in Presentation II} as
$
P(x)P(-x)Q(x)=1.
$
We interpret the polynomial $P(x)$ as the total Chern class of $\mathcal{U}$ and $Q(x)$ as the total Chern class of $\mathcal{U}^{\perp}/\mathcal{U}$. Now by using basic properties of Chern classes and short exact sequences \eqref{Eq.: Tautological S.E.S I} and \eqref{Eq.: Tautological S.E.S II} it is easy to see that the above relation does hold. 

The above discussion shows that we have a natural homomorphism of $\QQ$-algebras
\begin{align}\label{Eq.: algebra hom in the proof of pres II}
\psi \colon \QQ[a_1,a_2,b_1,\cdots,b_{n-2}]/(P(x)P(-x)Q(x) -1)  \to H^*(X, \QQ)
\end{align}   
sending $a_i$'s to $a_i$'s and $b_i$'s to $b_i$'s. To prove the proposition it is enough to establish two facts: i) $\sigma_i$'s can be expressed in terms of $a_i$'s and $b_i$'s, ii) the dimensions of both algebras in \eqref{Eq.: algebra hom in the proof of pres II} are equal.  To prove the first fact one can use again simple properties of Chern classes. The proof of the second fact is given below.

First, we need to show that $\QQ[a_1,a_2,b_1,\cdots,b_{n-2}]/(P(x)P(-x)Q(x) -1)$ is finite dimensional. This can be done similarly to the finite-dimensionality part of the proof of \cite[Theorem 1.2]{BKT}: essentially it is an application of \cite[Lemma 1.2]{BKT}. Second, we need to compute the dimension of this algebra. For this we can proceed similarly to the proof of \cite[Lemma 1.1]{BKT}, which is based on \cite{S}, and get the desired $2n(2n-1)$. 

\end{proof}

\subsection{Lines on $\IG(2,2n)$} 

Here we recall the description of lines on $X$ (see \cite{manivel-landsberg} for example). Since we have an inclusion $X \subset \G(2,2n)$, a line on $X$ is also a line on the ordinary Grassmannian $\G(2,2n)$. In particular it is given by a pair $(W_1,W_3)$ of nested subspaces of dimensions $1$ and $3$ respectively. The corresponding line being the set
$$\ell(W_1,W_3) = \{ V_2 \in X \ | \ W_1 \subset V_2 \subset W_3 \}.$$
In $X$, there are two different types of lines corresponding to two $\Sp_{2n}$-orbits in the variety $Y$ of all lines on $X$:
\begin{itemize}
\item If $W_3$ is not isotropic, then $W_1$ is the kernel of the symplectic form $\omega$ on $W_3$, therefore $W_1$ is determined by $W_3$. We shall denote by $\ell(W_3)$ a line of this type. The set of these lines is the open $\Sp_{2n}$-orbit in $Y$ which we will denote by $\Yo$. 
\item If $W_3$ is isotropic, then $W_1$ is any one-dimensional subspace in $W_3$. This is the closed $\Sp_{2n}$-orbit in $Y$.
\end{itemize}

\section{Small quantum cohomology of $\IG(2,2n)$}

\subsection{Two presentations}

As was already mentioned in Section \ref{SubSec.: Cohomology of IG}, the ordinary cohomology ring of the Grassmannian $\IG(2,2n)$ is generated by the special Schubert classes $\sigma_1, \dots , \sigma_{2n-2}$ with relations \eqref{Eq.: Eqns for H(IG), part I} and \eqref{Eq.: Eqns for H(IG), part II}. Since $\deg(q) = 2n-1$, the only relation that needs to be modified is the second equation in \eqref{Eq.: Eqns for H(IG), part II}. Moreover, up to a constant factor, this modification is unique for degree reasons. The complete answer for arbitrary Grassmannians $\IG(m,2n)$ was given in \cite[Theorem 1.5]{BKT} which we reproduce here in the special case of $m=2$.

\begin{thm}[Buch-Kresch-Tamvakis]
\label{thm-bkt}
The small quantum cohomology ring $\QH(X)$ is isomorphic to the quotient of the ring $K[\sigma_1,\cdots,\sigma_{2n-2}]$ by the ideal generated by the elements
$$\det(\sigma_{1 + j - i})_{1 \leq i,j \leq r}, \ \ \ \textrm{with } r \in [3,2n-2]$$
and the two elements 
$$\sigma_{n-1}^2 + 2 \sum_{i = 1}^{n - 1} (-1)^i \sigma_{n - 1 + i} \sigma_{n - 1 - i}  \ \ \textrm{  and  }\ \  \sigma_n^2 + 2 \sum_{i = 1}^{n - 2} (-1)^i \sigma_{n + i} \sigma_{n - i} + (-1)^{n+1} q \sigma_{1}.$$
\end{thm}

\medskip

Combining the above theorem with Proposition \ref{Prop.: Presentation for coh II} we arrive at the following statement.
\begin{cor}\label{Cor.: Presentation for small QH}
The small quantum cohomology ring $\QH(X)$ is isomorphic to the quotient of the ring $K[a_1,a_2,b_1,\cdots,b_{n-2}]$ by the ideal generated by 
\begin{align}\label{Eq.: Relations for small QH II}
(1 - (2a_2 - a_1) x^2 + a_2^2 x^4)(1 + b_1 x^2 + \cdots + b_{n-2}
  x^{2n-4}) = 1 + q a_1 x^{2n}.
\end{align}


\end{cor}

\subsection{Structure of $\QH(X)$}
\label{SubSec.: Structure of QH}

In this paragraph we will study the decomposition of the $K$-algebra $\QH(X)$ into the direct product, or, equivalently, the decomposition of the $K$-scheme $\spec(\QH(X))$ into connected components.

We will use the presentation of $\QH(X)$ described in Corollary \ref{Cor.: Presentation for small QH}. Thus, the $K$-scheme $\spec(\QH(X))$ is given as a closed subscheme of the affine space $\mathbb{A}^n =\spec (K[a_1,a_2, b_1, \dots, b_{n-2}])$ defined by the equations
\begin{align}\label{Eq.: System for QH(IG(2,2n))} \notag
& 2a_2 -a_1^2 + b_1 = 0 \\ \notag
& a_2^2 + b_1(2a_2 -a_1^2)+b_2 =0 \\ \notag
& b_1a_2^2 + b_2(2a_2 -a_1^2)+b_3 =0 \\ 
& \dots \\ \notag
& b_{i-1} a_2^2 + b_i(2a_2 -a_1^2) + b_{i+1} =0 \\ \notag
& \dots \\ \notag
& b_{n-4} a_2^2 + b_{n-3}(2a_2 -a_1^2) + b_{n-2} =0 \\ \notag
& b_{n-3} a_2^2 + b_{n-2}(2a_2 -a_1^2)  =0 \\ \notag
& b_{n-2}a_2^2 - qa_1 =0 .
\end{align}
It is clear that the origin of $\mathbb{A}^n$ is a solution of this system. Moreover, this solution corresponds to a fat point of $\spec(\QH(X))$. Indeed, it is easy to see that the Zariski tangent space of \eqref{Eq.: System for QH(IG(2,2n))} at the origin is one-dimensional. Thus, the origin is a fat point of $\spec(\QH(X))$. Let $A$ be the corresponding factor of $\QH(X)$. Thus, we have the direct product decomposition
\begin{align}\label{Eq.: Direct product decomposition for QH}
\QH(X) = A \times B,
\end{align}
where $B$ corresponds to components of $\spec(\QH(X))$ supported outside of the origin.

To determine the structure of $A$ we look for solutions of \eqref{Eq.: System for QH(IG(2,2n))} in the ring $K[\varepsilon]/\varepsilon^{n-1}$ that extend the zero solution. It is easy to see that putting $a_1=0$, letting $a_2$ be an arbitrary element of $K[\varepsilon]/\varepsilon^{n-1}$ of the form $O(\varepsilon)$, and recovering $b_i$'s by elimination  process gives rise to a solution of \eqref{Eq.: System for QH(IG(2,2n))} with the property $b_i = O(\varepsilon^i)$. Another way to phrase it is to say that we have obtained a surjective algebra homomorphism 
\begin{align}\label{Eq.: ring homo origin}
A \to K[\varepsilon]/\varepsilon^{n-1}.
\end{align}
Thus, the dimension of $A$ is at least $n-1$. Moreover, below we will see that this map is an isomorphism.

\smallskip

Now let us examine the structure of $B$, i.e. we need to study
solutions of \eqref{Eq.: System for QH(IG(2,2n))} different from the
origin\footnote{Note that from \eqref{Eq.: System for QH(IG(2,2n))} it
  is clear that variables $b_i$ can be eliminated, i.e. they are
  determined by $a_i$'s.}. It is convenient to rewrite \eqref{Eq.:
  Relations for small QH II} as 
\begin{align*}
(z^4 - (2a_2 - a_1) z^2 + a_2^2)(z^{2n-4} + b_1 z^{2n-6} + \cdots +
  b_{n-2}) = z^{2n} - a_1,
\end{align*}
where we set $q = -1$ for convenience. 
By making the substitution $a_1 = z_1 + z_2, a_2 = z_1 z_2$, and putting $Q(z)=z^{2n-4} + b_1 z^{2n-6} + \cdots + b_{n-2}$ we arrive at
\begin{align}\label{Eq.: equation in z}
(z^2 -z_1^2)(z^2 - z_2^2) Q(z) = z^{2n} - (z_1 + z_2).
\end{align}
In geometric terms this manipulation corresponds to pulling back our system \eqref{Eq.: System for QH(IG(2,2n))} with respect to the morphism
\begin{align*}
& \spec (K[z_1,z_2, b_1, \dots, b_{n-2}]) \to \spec (K[a_1,a_2, b_1, \dots, b_{n-2}]) \\
\intertext{defined by}
& a_1 \mapsto z_1 +z_2, \quad   a_2 \mapsto z_1z_2 ,   \quad    b_i \mapsto b_i.
\end{align*}
It is a double cover unramified outside of the locus $z_1=z_2$.

Let us count solutions of \eqref{Eq.: equation in z} for which $z_1 \neq z_2$ and both of them are non-zero. This reduces to counting pairs $z_1, z_2$ satisfying
\begin{align*}
& z_1^{2n}=z_1+z_2 \\
& z_2^{2n}=z_1+z_2.
\end{align*}
Eliminating $z_2$ using the first equation we obtain that $z_1$ must be a solution of
$$
(z_1^{2n}-z_1)^{2n}= z_1^{2n}.
$$
Now it is straightforward to count that there are $2(n-1)(2n-1)$ distinct pairs of numbers $(z_1, z_2)$ satisfying our conditions that we obtain in this way. 

In terms of the original system \eqref{Eq.: System for QH(IG(2,2n))} the above computation means that there are at least $(n-1)(2n-1)$ distinct solutions outside of the origin. Note that we have divided the number of solutions by 2, since the initial count was on the double cover. In other words the dimension of $B$ is at least $(n-1)(2n-1)$.

\smallskip

Up to now we have shown that $\dim_K (A) \geq n-1$ and  $\dim_K (B) \geq (n-1)(2n-1)$. Since $\dim_K(\QH(X)) = 2n(n-1)$, this implies that $\dim_K (A) = n-1$ and  $\dim_K (B) = (n-1)(2n-1)$. Hence, \eqref{Eq.: ring homo origin} is an isomorphism.  

\smallskip

The above discussion can be summarized in the following statement.

\begin{prop}
The scheme $\spec (\QH(X))$ decomposes into the disjoint union of $(2n-1)(n-1)$ reduced points $\spec(K)$ and one fat point $\spec(K[\varepsilon]/\varepsilon^{n-1})$.
\end{prop}

\section{Four-point Gromov-Witten invariants}

In Section \ref{Sec.: BQH} we will study the deformation of $\QH(X)$ in the big quantum cohomology $\BQH(X)$ in a very special direction, namely in the direction $\sigma_2$. In particular we shall need to compute the following invariants:
$$I_1(\pt,\sigma_2,\sigma_i,\sigma_j).$$
Note first that for dimension reasons these invariants vanish except if we have $i + j = 2n-2$. We shall also use the following result which simply follows from Kleiman-Bertini's Theorem \cite{kleiman} (see also \cite[Lemma 14]{fulton-pandharipande}).

\begin{lemma}
The Gromov-Witten invariant $I_1(\pt,\sigma_2,\sigma_i,\sigma_j)$ is the number of lines meeting general representatives of the cohomology classes $\pt$, $\sigma_2$, $\sigma_i$ and $\sigma_j$. Furthermore, given any open dense subset of the set of lines, the above lines can be chosen in this open subset.
\end{lemma}

As a consequence, we may only consider lines in $\Yo$, the open orbit of the variety $Y$ of lines on $X$ in order to compute the above invariants. For $Z \subset X$ a subvariety in $X$, we define
$$\Yo(Z) = \{ \ell \in \Yo \ | \ \ell \cap Z \neq \emptyset \}.$$

\begin{lemma}
We have the following equalities
$$\begin{array}{l}
\Yo(\{E_2\}) = \{ \ell(W_3) \in \Yo \ | \ E_2 \subset W_3 
\} \simeq \p(\C^{2n}/E_2) \setminus \p(E_2^\perp/E_2) \\
\Yo(Z(E_{2n - k -1})) = \{ \ell(W_3) \in \Yo \ | \ \dim(W_3 \cap E_{2n - k - 1}) \geq 1 \}. \\
\end{array}$$
For $E_2$ and $E_{2n - k - 1}$ in general position, their intersection is isomorphic to
$$\p(E_2 + E_{2n - k - 1}/E_2) \setminus \p((E_2 + E_{2n - k - 1}) \cap E_2^\perp/E_2)\subset \p(\C^{2n}/E_2) \setminus \p(E_2^\perp/E_2).$$
\end{lemma}

\begin{proof}
The first equality is clear. For the isomorphism we just map $W_3$ to $W_3/E_2$ and remark that $W_3$ is non isotropic if and only if $W_3/E_2$ is not contained in $E_2^\perp/E_2$. 

The second equality is also easy. Let $V_2 \in \ell(W_3) \cap Z(E_{2n - k -1})$. Then $\dim(V_2 \cap E_{2n - k - 1}) \geq 1$ and $V_2 \subset W_3$. Conversely, for $W_3$ satisfying $\dim(W_3 \cap E_{2n - k - 1}) \geq 1$, we can find a $2$-dimensional isotropic subspace $V_2 \subset W_3$ meeting $E_{2n - k - 1}$ non trivially.

For $k = 0$, we have $E_2 + E_{2n - k - 1} = \C^{2n}$ (recall that $E_2$ and $E_{2n - k - 1}$ are in general position) thus $\Yo(Z(E_{2n - k -1})) = \Yo$ and the result follows. For $k \geq 1$, we have $E_2 \cap E_{2n - k - 1} = 0$ (since they are in general position). This implies for $W_3$ in the intersection the inclusion $W_3 \subset E_2 + E_{2n - k - 1}$ thus $W_3/E_2 \subset E_2 + E_{2n - k - 1}/E_2$ proving the result. 
\end{proof}

\begin{cor}\label{Cor.: 4-point invariant}
We have $I_1(\pt,\sigma_2,\sigma_i,\sigma_j) = \delta_{i + j,2n-2}.$
\end{cor}

\begin{proof}
We already explained that the invariant vanishes unless $i + j = 2n-2$. In that case, $I_1(\pt,\sigma_2,\sigma_i,\sigma_j)$ is the number of lines meeting $\{E_2\}$, $Z(\{E_{2n - 3}\})$, $Z(\{E_{2n - i - 1}\})$ and $Z(\{E_{2n - j - 1}\})$ where the spaces $E_k$ are in general position and such that $\omega\vert_{E_k}$ has minimal rank. The variety of such lines is the intersection in $\p(\C^{2n}/E_2) \setminus \p(E_2^\perp/E_2)$ of the three linear spaces
$$\p(E_2 + E_{2n - 3}/E_2), \ \ \p(E_2 + E_{2n - i - 1}/E_2) \ \ \textrm{ and } \ \  \p(E_2 + E_{2n - j - 1}/E_2).$$
These linear spaces are in general position and of respective codimension $1$, $i-1$ and $j-1$. These codimensions add up to the dimension of $\p(\C^{2n}/E_2)$ thus the linear spaces meet in exactly one point.
\end{proof}

\section{Big quantum cohomology of $\IG(2,2n)$}
\label{Sec.: BQH}

In \cite{GMS,Pe} it is proved that the big quantum cohomology of $\IG(2,2n)$ is generically semisimple. In this section we will give an alternative proof of this result. This is the main result of this article.

As before we let $X=\IG(2,2n)$ and consider the deformation $\BQH_\tau(X)$ of the small quantum cohomology $\QH(X)$ inside the big quantum cohomology $\BQH(X)$ in the direction $\tau = \sigma_2$. Explicitly it means the following. For any cohomology classes $a,b \in H^*(X)$ the product is of the form
$$ a \star_\tau b = a \starz b + t \sum_{\sigma}\sum_{d \geq 1} q^d I_d(a,b,\sigma,\sigma_2) \sigma^\vee + O(t^2),$$
where $a \starz b$ is the small quantum product, $\sigma$ runs over the basis of the cohomology consisting of Schubert classes, $\sigma^\vee$ is the dual basis, and $t$ is the deformation parameter.

According to the dimension axiom for GW invariants $I_d(\sigma_i,\sigma_j,\sigma,\sigma_2)$ vanishes unless 
$$
i + j + \deg(\sigma) + 2 =  (2n-1)d + 2(2n-2).
$$ 
From here one easily deduces that $d = 1$ is the only possibility unless $i = j = 2n-2$. Therefore, applying Corollary \ref{Cor.: 4-point invariant}, we have
\begin{align}\label{Eq.: BQH product for special classes}
\sigma_i \star_\tau \sigma_j = \sigma_i \starz \sigma_j +  \delta_{i+j,2n-2} qt + O(t^2)
\end{align}
for $i + j < 4n-4$.

Consider the following elements in $\BQH_\tau(X)$
$$\begin{array}{l}
\Delta_r = det(\sigma_{1 + j - i})_{1 \leq i,j \leq r} \ \ \ \ \ \ \ \textrm{for $r \in [3,2n-2]$} \\
\\
\displaystyle{\Sigma_{2n-2} = \sigma_{n-1} \star_\tau \sigma_{n-1} + 2 \sum_{i = 1}^{n - 1} (-1)^i \sigma_{n - 1 + i} \star_\tau \sigma_{n - 1 - i}} \\
\\
\displaystyle{\Sigma_{2n} = \sigma_n \star_\tau \sigma_n + 2 \sum_{i = 1}^{n - 2} (-1)^i \sigma_{n + i} \star_\tau \sigma_{n - i} + (-1)^{n+1} q \sigma_{1},}\\
\end{array}
$$
where all products are taken in $\BQH_\tau(X)$. Note that these are ``the same" elements as those defining relations in the  presentation of $\QH(X)$ in Theorem \ref{thm-bkt}. The lower indices indicate the degree of the respective element in $\BQH_\tau(X)$. The variable $t$ can only appear together with a positive power of $q$ so that $q^dt$ can only occur in elements of degree at least $\deg(q) + \deg(t) = 2n-2$.

\begin{lemma}
We have $\Delta_r = O(t^2)$ for all $r \in [3,2n-2]$,  $\Sigma_{2n-2}
= (-1)^n qt + O(t^2)$ and $\Sigma_{2n} = O(t^2)$.
\end{lemma}

\begin{proof}

(i) Here we prove the statement about $\Delta_r$'s. From Theorem \ref{thm-bkt} we know that $\Delta_r = O(t^2)$ for $r \in [3,2n-3]$. Thus, we only need to consider $\Delta_{2n-2}$. Inductively developing the determinant with respect to the first column we have
\begin{align*}
&\Delta_{2n-2} = \sum_{s=1}^{2n-2} (-1)^{s-1} \sigma_s \star_\tau \Delta_{2n-2-s} = O(t^2) - \sigma_{2n-4} \star_\tau \Delta_2 + \sigma_{2n-3} \star_\tau \Delta_1 - \sigma_{2n-2},
\end{align*}
where we used the fact that $\Delta_r = O(t^2)$ for $r \in [3,2n-3]$. Thus, we need to prove that 
\begin{align}\label{Eq.: formula I in the proof of Lemma O(t^2)}
\sigma_{2n-4} \star_\tau \sigma_2  - \sigma_{2n-4} \star_\tau \sigma_1 \star_\tau \sigma_1 + \sigma_{2n-3} \star_\tau \sigma_1  - \sigma_{2n-2}
\end{align}
is of the form $O(t^2)$. We will see this by reducing everything to products of special Schubert classes. First, let us look at the term $\sigma_{2n-4} \star_\tau \sigma_1 \star_\tau \sigma_1$. By degree reasons (or dimension axiom for GW invariants) we have
\begin{align}\label{Eq.: formula II in the proof of Lemma O(t^2)}
(\sigma_{2n-4} \star_\tau \sigma_1 ) \star_\tau \sigma_1 = ( \sigma_{2n-3} + \sigma_{2n-3}' ) \star_\tau \sigma_1,
\end{align}
where $\sigma_{2n-3}'$ is a Schubert class of degree $2n-3$ different from $\sigma_{2n-3}$. Since by \eqref{Eq.: BQH product for special classes} we have $\sigma_{2n-3} \star_\tau \sigma_1 =  \sigma_{2n-3} \starz \sigma_1 + qt + O(t^2)$, we only need to take care of the term  $\sigma_{2n-3}' \star_\tau \sigma_1$. By degree reasons we have
\begin{align*}
&\sigma_{2n-3}' \star_\tau \sigma_1 = \sigma_{2n-3}' \starz \sigma_1 +I_1(\sigma_{2n-3}',\sigma_1,\pt ,\sigma_2) qt + O(t^2).            \end{align*}
Applying the dimension axiom for GW invariants we see that the 4-point invariant $I_1(\sigma_{2n-3}',\sigma_1,\pt ,\sigma_2)$ is equal to the 3-point invariant $I_1(\sigma_{2n-3}',\pt ,\sigma_2)$. The latter can be computed using results in \cite{ChPe} and we get $I_1(\sigma_{2n-3}',\pt ,\sigma_2) = 1$. Therefore, we obtain
\begin{align}\label{Eq.: formula III in the proof of Lemma O(t^2)}
&\sigma_{2n-3}' \star_\tau \sigma_1 = \sigma_{2n-3}' \starz \sigma_1 + qt + O(t^2).            
\end{align}
Plugging \eqref{Eq.: formula II in the proof of Lemma O(t^2)} and \eqref{Eq.: formula III in the proof of Lemma O(t^2)} into \eqref{Eq.: formula I in the proof of Lemma O(t^2)} and using that $\Delta_{2n-2}$ vanishes modulo $t$ we get the required statement.

\medskip

(ii) To prove the statements about $\Sigma_{2n-2}$ and $\Sigma_{2n}$ we only need to use \eqref{Eq.: BQH product for special classes} and the fact that these elements vanish modulo $t$.
\end{proof}

\subsection{Structure of $\BQH_\tau(X)$}

Recall from Section \ref{SubSec.: Structure of QH} that for $\QH(X)$ we have the direct product decomposition 
\begin{align*}
& \QH(X)= A \times B = A \times \prod_{i \in I} B_i
\end{align*}
where $A=K[\varepsilon]/\varepsilon^{n-1}$ each $B_i$ is the ground
field $K$. By Hensel's lemma this decomposition lifts to $\BQH_\tau(X)$ and we have
\begin{align}\label{Eq.: Direct product decomposition for BQH_tau}
& \BQH_\tau(X) = \mathcal{A} \times  \prod_{i \in I} \mathcal{B}_i
\end{align}
where $\mathcal{A}$ and $\mathcal{B}_i$ are algebras over $K[[t]]$. 

\smallskip

By construction $K[[t]]$-algebra $\BQH_\tau(X)$ is a free module over $K[[t]]$ of finite rank equal to $\dim H^*(X)$. Therefore, $\cA$ and $\cB_i$'s are also free $K[[t]]$-modules of finite rank. Reducing modulo $t$ one sees that the rank of $\cA$ is $n-1$ and the ranks of $\cB_i$'s are all equal to 1. Similarly we obtain that the natural algebra homomorphism $K[[t]] \to \cB_i$ is an isomorphism of $K$-algebras.

\begin{prop}\label{Prop.: Regularity of BQH}
The ring $\BQH(X)$ is regular.
\end{prop}

\begin{proof}

Let us first prove that $\BQH_\tau(X)$ is a regular ring, which reduces to showing that $\cA$ is regular. For this we use the presentation of $\BQH_\tau(X)$ as the quotient of $(K[[t]])[\sigma_1,\cdots,\sigma_{2n-2}]$ with respect to the ideal generated by the elements $\Delta_r$ with $r \in [3,2n-2],$  $\Sigma_{2n-2}$ and $\Sigma_{2n}$. Computing the Jacobian of this system at the origin one sees that it is of maximal possible rank $2n-2$. From here, using Prop.~8.4A and Prop.~8.7 of \cite{hartshorne}, as usual, one obtains that $\dim_K \mu/\mu^2 = 1$, where $\mu$ is the maximal ideal of $\cA$. Since the Krull dimension of $\cA$ is also equal to one, we obtain the regularity of $\cA$.

The regularity of $\BQH(X)$ can be obtained by using an identical argument with the Jacobian. The only difference is the number of the deformation parameters. Note that no additional GW invariants are necessary.
\end{proof}

\begin{thm}\label{Thm.: semisimplicity}
$\BQH(X)$ is generically semisimple.
\end{thm}

\begin{proof}
This is a simple corollary of the above Proposition \ref{Prop.: Regularity of BQH}. Let us give a short proof for completeness.

The coordinate ring of the generic fiber 
\begin{align*}
\BQH(X)_{\eta} = \BQH(X) \otimes_{K[[t_0, \dots, t_s]]} K((t_0, \dots, t_s))
\end{align*}
is a localization of $\BQH(X)$. Therefore, it is also a regular ring, since $\BQH(X)$ was regular. This implies that $\BQH(X)_{\eta}$ is a product of finite field extensions of $K((t_0, \dots, t_s))$ which was our definition of semisimplicity. 
\end{proof}

\begin{rmk}
We view the above theorem as a consequence of "generic smoothness" in characteristic zero. The reason it is not formulated this way lies in the fact that both $\BQH(X)$ and $K[[t_0, \dots, t_s]]$ have modules of K\"ahler differentials of infinite rank. It might be possible to bypass this obstacle by working with a version of K\"ahler differentials that uses the topology of these rings but we did not pursue this. 
\end{rmk}


\begin{thebibliography}{10}

\bibitem{BKT} A.S.~Buch, A.~Kresch and H.~Tamvakis,  
\emph{Quantum Pieri rules for isotropic Grassmannians}. Invent. Math. {\bf 178} (2009), no. 2, 345--405. 

\bibitem{ChMaPe} P.-E.~Chaput, L.~Manivel, N.~Perrin. {\it Quantum cohomology of minuscule homogeneous spaces III : semisimplicity and consequences.} Canad. J. Math. 62 (2010), no. 6, 1246--1263.

\bibitem{ChPe} P.-E.~Chaput, N.~Perrin. {\it On the quantum cohomology of adjoint varieties.} Proc. Lond. Math. Soc. (3) 103 (2011), no. 2, 294--330. 

\bibitem{fulton-pandharipande} W.~Fulton, R.~Pandharipande, {\em Notes on stable maps and quantum cohomology}. Algebraic geometry—Santa Cruz 1995, 45--96,
Proc. Sympos. Pure Math., 62, Part 2, Amer. Math. Soc., Providence, RI, 1997.

\bibitem{GaGoIr} S.~Galkin, V.~Golyshev, H.~Iritani. {\it Gamma classes and quantum cohomology of Fano manifolds: Gamma conjectures.} arXiv:1404.6407

\bibitem{GMS} S.~Galkin, A.~Mellit, M.~Smirnov, \emph{Dubrovin's conjecture for $\IG(2,6)$}. IMRN 2014, doi: 10.1093/imrn/rnu205.

\bibitem{hartshorne} R.~Hartshorne, \emph{Algebraic geometry}. Graduate Texts in Mathematics, No. 52. Springer-Verlag, New York-Heidelberg, 1977.

\bibitem{HeMaTe} C.~Hertling, Yu.~Manin, C.~Teleman. {\it An update on semisimple quantum cohomology and F-manifolds.} Proc. Steklov Inst. Math. 264 (2009), no. 1, 62--69
 
\bibitem{kleiman} S.L.~Kleiman, {\em The transversality of a general translate}.
Compositio Math. {\bf 28} (1974), 287--297. 

\bibitem{manivel-landsberg} J.M.~Landsberg and L.~Manivel, {\em On the projective geometry of rational homogeneous varieties}. Comment. Math. Helv. {\bf 78} (2003), no. 1, 65--100.


\bibitem{Ma} Yu.~I.~Manin. {\it Frobenius manifolds, quantum cohomology, and moduli spaces.} AMS Colloquium Publ. 47, Providence RI, 1999, 303 pp.


\bibitem{Pe} N.~Perrin, {\it Semisimple quantum cohomology of some Fano varieties.} arXiv:1405.5914

\bibitem{S} R.~Stanley, {\it Hilbert functions of graded algebras.} Adv. Math. 28, 57--83 (1978)



\end{thebibliography}
\end{document}